\documentclass[12pt]{article}
\usepackage{amsmath,amsthm,amssymb,mathrsfs,graphicx,subfigure,extarrows,float}
\usepackage[numbers]{natbib}
\usepackage{titlesec}
\usepackage[titletoc,toc,title]{appendix}
\usepackage{graphicx}
\usepackage{faktor}
\usepackage{xcolor}
\usepackage[a4paper, textheight=9.6in]{geometry}
\usepackage{pgf}
\usepackage{tikz-cd}

\title{On the Representation theory of the Infinite Temperley-Lieb algebra}
\author{Stephen T. Moore \\
	\multicolumn{1}{p{.8\textwidth}}{\centering\emph{Dept. of Mathematics, Ben Gurion University, Beer-Sheva, Israel,\\ stm862@gmail.com}}}

\newtheorem{thm}{Theorem}[section]
\newtheorem{prop}{Proposition}[section]
\newtheorem{defin}{Definition}[section]
\newtheorem{conj}{Conjecture}[section]
\newtheorem{corr}{Corollary}[section]
\newtheorem{lemma}{Lemma}[section]
\newtheorem{remark}{Remark}[section]

\begin{document}
	\maketitle
\begin{abstract}
We begin the study of the representation theory of the infinite Temperley-Lieb algebra. We fully classify its finite dimensional representations, then introduce infinite link state representations and classify when they are irreducible or indecomposable. We also define a construction of projective indecomposable representations for $TL_{n}$ that generalizes to give extensions of $TL_{\infty}$ representations. Finally we define a generalization of the spin chain representation and conjecture a generalization of Schur-Weyl duality.
\end{abstract}

\section{Introduction}
The Temperley-Lieb algebras \cite{TL} are a family of finite dimensional algebras that first appeared in relation to statistical mechanics models, before being rediscovered by Jones in the standard invariant of subfactors \cite{Jones1}, who used them to define the Jones polynomial knot invariant \cite{Jones2}. They were further shown to be the Schur-Weyl dual of the quantum group $U_{q}(\mathfrak{sl}_{2})$ by Jimbo and Martin \cite{Jimbo, MartinSW}. The representation theory of the Temperley-Lieb algebra was first studied by Martin \cite{MartinBook} who used combinatorial techniques to classify irreducibles and constructed projective representations as ideals generated by projections. Shortly after, Goodman and Wenzl \cite{GW} studied path idempotents on Bratelli diagrams to classify the blocks of the algebra. A more recent study was done using cellular algebra techniques by Ridout and Saint-Aubin \cite{RSA} based on earlier work by Westbury \cite{Westbury}. A follow-up paper by Bellet\^{e}te, Ridout, and Saint-Aubin \cite{BRSA} fully classified the indecomposable representations of the Temperley-Lieb algebra.

Whilst current research has focused on the finite dimensional Temperley-Lieb algebras, their infinite dimensional generalization also has potential applications. It has been conjectured based on physical considerations that in the $n\rightarrow\infty$ limit, the representation theory of the Temperley-Lieb algebra should approach that of the Virasoro algebra. An attempt to realize this based on an inductive limit of $Rep TL_{n}$ was given in \cite{GS}, however we shall see that there are large families of irreducible $TL_{\infty}$ representations not appearing in this construction. The infinite Temperley-Lieb algebra has also appeared in the study of the periplectic Lie superalgebra \cite{EntovaETAL} where a categorical action was defined.

The paper is organized as follows: In Section \ref{Background} we review the representation theory of the finite Temperley-Lieb algebras. In Section \ref{TL infty} we begin the study of the infinite Temperley-Lieb algebra, and fully classify its finite dimensional representations in Theorem \ref{fin dim reps}. We then introduce \textit{infinite link state representations}, which generalize standard $TL_{n}$ representations and can be considered analogues of Verma modules, and classify when these are irreducible and isomorphic in Theorems \ref{Irreducibility of representations} and \ref{Irr iso}, as well as showing that they have at most two irreducible subquotients. Next in Definition \ref{Extensions} we define a construction of projective indecomposable representations for $TL_{n}$ that generalizes to give extensions of certain $TL_{\infty}$ representations. Finally in Section \ref{spin chain} We relate a bilinear form defined based on cellular algebras to one defined on the spin chain representation. We also define a generalization of the spin chain representation which gives a functor $\mathcal{F}:Rep U_{q}(\mathfrak{sl}_{2})\rightarrow Rep TL_{\infty}$. We conjecture this functor defines an equivalence of categories between the category of finite dimensional $U_{q}(\mathfrak{sl}_{2})$ representations and certain subcategories of $Rep TL_{\infty}$.

\section{The Representation theory of $TL_{n}$.}\label{Background}
The Temperley-Lieb algebra $TL_{n}(\delta)$ is a finite dimensional algebra with generators $\{1,e_{1},...,e_{n-1}\}$ and relations:
\begin{align*}
e_{i}^{2}&=\delta e_{i}\\ 
e_{i}e_{i\pm 1}e_{i}&=e_{i}\\ 
e_{i}e_{j}&=e_{j}e_{i}, \:\:\: \lvert i-j\rvert>1
\end{align*}
with parameter $\delta:=q+q^{-1}$, $q\in\mathbb{C}\setminus \{0\}$. This algebra has a natural diagrammatic description, due to Kauffman \cite{Kauffman}, where basis elements of the algebra consist of rectangles with $n$ marked points along the top and bottom, and non-intersecting strings connecting these points. Multiplication is given by vertical concatenation of diagrams, and removing a closed loop corresponds to multiplying by $\delta$. In this description, the identity element is a diagram with $n$ vertical strings, and the generator $e_{i}$ is given by a cup connecting the $i$th and $(i+1)$th points along the top, a cap connecting the same points along the bottom, and vertical strings everywhere else. We will sometimes need to use certain polynomials in $q$, known as \textit{quantum integers}:
\begin{defin}
The quantum integer $[n]$ is given by $[n]:=\frac{q^{n}-q^{-n}}{q-q^{-1}}$. Alternatively, $[n]=\sum\limits_{i=0}^{n}q^{n-2i-1}$, or $[n]=\delta[n-1]-[n-2]$, with $[1]=1$ and $[2]=\delta$. We denote $[n]!:=1\times[2]\times...\times[n]$.
\end{defin} 
\begin{figure}[H]
	\centering
	\includegraphics[width=0.6\linewidth]{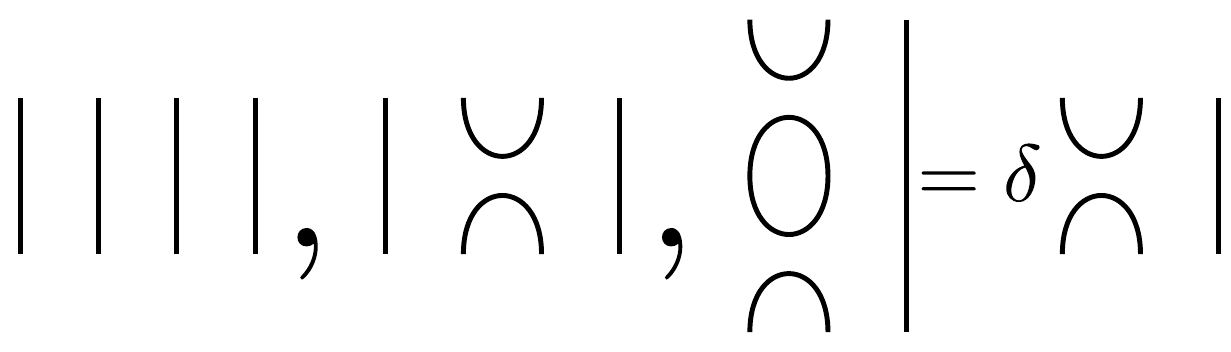}
	\caption{The elements $1,e_{2}\in TL_{4}$, and the relation $e_{1}^{2}=\delta e_{1}$.}
\end{figure}
The representation theory of the Temperley-Lieb algebras was first developed by Martin \cite{MartinBook} using combinatorial techniques, and Goodman and Wenzl \cite{GW} based on path idempotents for Bratelli diagrams. However for our purposes we follow the more recent version by Ridout and Saint-Aubin \cite{RSA} based on cellular algebra methods.\\

\begin{defin} A (n-p)-link state, with $0\leq p\leq\lfloor\frac{n}{2}\rfloor$, is a line with $n$ marked points on it, labelled left to right from $1$ to $n$, with non-intersecting strings attached to each point, such that $p$ strings connect two points together, and the other $n-2p$ strings are only connected to a single point.
\end{defin}
We refer to the strings connected to two points as \textit{cups}, and the strings connected to a single points as a \textit{string}. A cup connecting the points $i$ and $i+1$ is referred to as a \textit{simple cup} at the $i$th point.
\begin{figure}[H]
	\centering
	\includegraphics[width=0.3\linewidth]{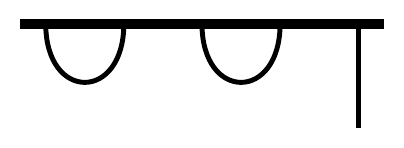}
	\caption{A $(5,2)$ link state.}
\end{figure}
The number of $(n,p)$-link states is given by $d_{n,p}:=\left(\begin{array}{c}n\\ p\end{array}\right)-\left(\begin{array}{c}n\\ p-1\end{array}\right)$.\\
\begin{defin}
The $TL_{n}(\delta)$ standard representation $\mathcal{V}_{n,p}$ is the representation with basis given by the set of $(n,p)$-link states, with $TL_{n}$ acting by concatenation, where the action of a $TL_{n}$ generator is zero if it increases the number of cups of the link state.
\end{defin}
There is a symmetric invariant bilinear form $\langle \cdot,\cdot \rangle_{n,p}$ on $\mathcal{V}_{n,p}$ defined as follows. Reflect the left entry along its line then join on top of the right entry. It is zero if every string on one side isn't connected to a string on the other. Otherwise it gives $\delta^{k}$, where $k$ is the number of closed loops. For example, the bilinear form on $\mathcal{V}_{4,1}$ is as follows:
\begin{figure}[H]
	\centering
	\includegraphics[width=0.4\linewidth]{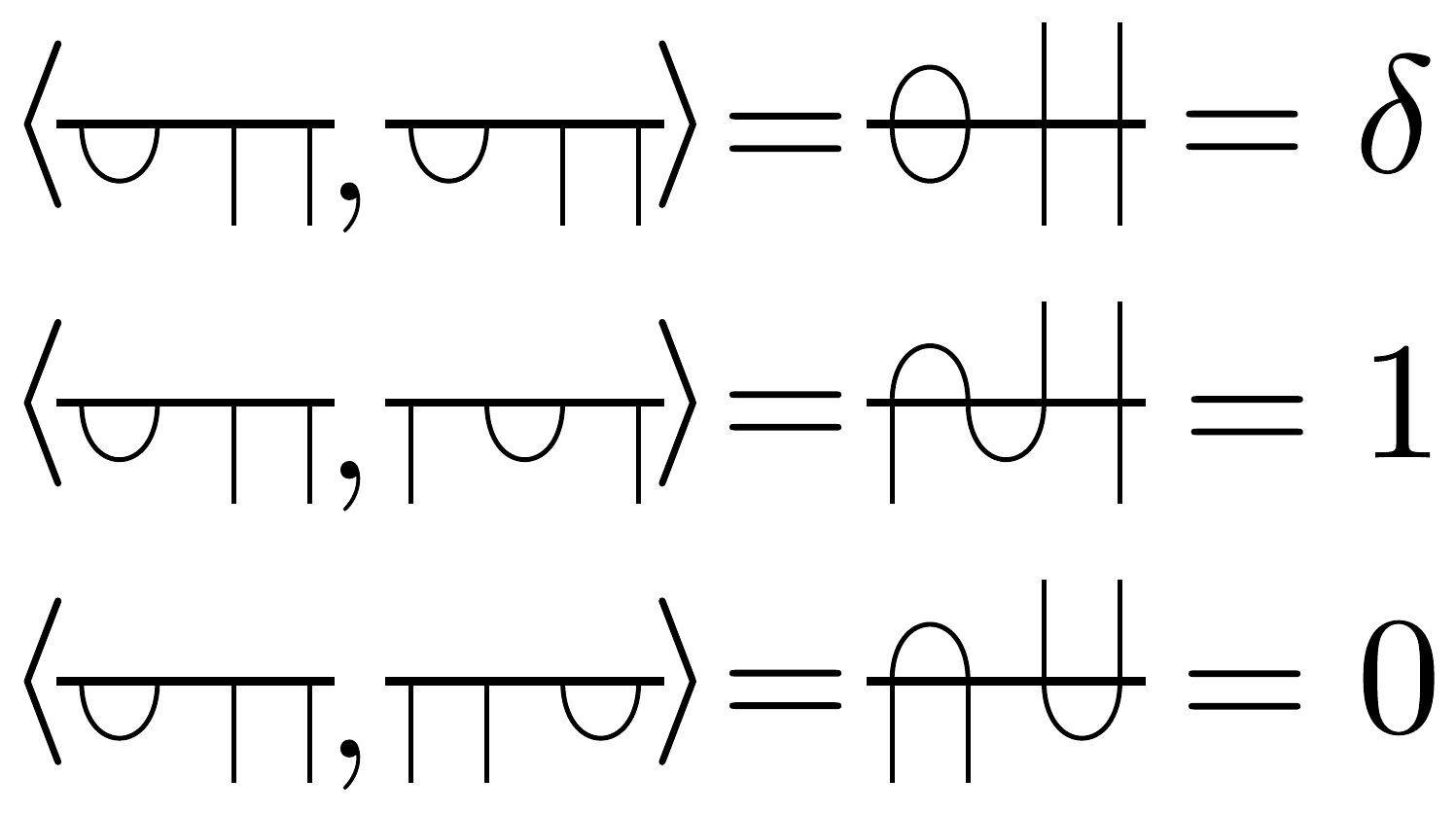}
\end{figure}
Using this bilinear form, we can define two further families of $TL_{n}$ representations:
\begin{defin}
The kernel of the bilinear form, $\mathcal{R}_{n,p}:=\{x\in\mathcal{V}_{n,p}: \langle x,y\rangle=0\:\:\forall y\in\mathcal{V}_{n,p}\}$ is a subrepresentation of $\mathcal{V}_{n,p}$, called the \textit{radical}. We denote the quotient representation	by $\mathcal{L}_{n,p}:=\mathcal{V}_{n,p}/\mathcal{R}_{n,p}$.
\end{defin}
Given $(n,p)$-link states $x,y$, we can form the $TL_{n}$ element $\lVert x y\rVert$ by reflecting $y$ and joining it to the bottom of $x$. This leads to the following lemma:
\begin{lemma}\label{lemma 1}
For $x,y,z\in\mathcal{V}_{n,p}$, $\lVert x y\rVert z=\langle y,z\rangle x$.
\end{lemma}	
This can be viewed as saying that if $x,z\notin\mathcal{R}_{n,p}$, then there are $TL_{n}$ elements mapping $x$ to $z$ and $z$ to $x$. This lemma is essential for the proof of the following proposition:
\begin{prop}
If $\langle\cdot,\cdot\rangle_{n,p}$ is not identically zero, then $\mathcal{V}_{n,p}$ is cyclic and indecomposable, and $\mathcal{L}_{n,p}$ is irreducible.
\end{prop}	
The only case in which the bilinear form is identically zero is when $\delta=0$ and $n=2p$, so in this case $\mathcal{R}_{2p,p}=\mathcal{V}_{2p,p}$. In general, the radical is described by the following:
\begin{prop}
The radical $\mathcal{R}_{n,p}$ is irreducible.
\end{prop}
Hence the standard representations are either irreducible or indecomposable	with two irreducible components. Determining when the radical is non-trivial is done by studying the Gram matrix of the bilinear form. Namely, take the matrix whose entries are given by $\langle x,y\rangle$ for link states $x,y\in\mathcal{V}_{n,p}$. Then when the determinant of this matrix is zero, the radical is non-trivial. The result of this can be summarized as follows:
\begin{thm}
Let $l\in\mathbb{N}$ be minimal such that $q^{2l}=1$.
\begin{itemize}
\item If no such $l$ exists, $l=1$, or $l=2$ and $n$ is odd,  then $\mathcal{R}_{n,p}=0$.
\item If $n-2p+1=kl$ for some $k\in\mathbb{N}$, $k\neq 0$, then $\mathcal{R}_{n,p}=0$.
\item If $l=2$ (i.e. $\delta=0$) and $n=2p$, then $\mathcal{R}_{2p,p}=\mathcal{V}_{2p,p}$.
\item If $0\leq n-2p+1<l$ then $\mathcal{R}_{n,p}=0$.
\item Otherwise $\mathcal{R}_{n,p}\neq 0,\mathcal{V}_{n,p}$.	
\end{itemize}
\end{thm}
Generally, the representation theory of $TL_{n}$ can be split into the cases of $q$ being \textit{generic} and $q$ a \textit{root of unity}, where we consider $q=\pm 1$ as being generic. For roots of unity, we refer to $(n,p)$ as being \textit{critical} whenever $n-2p+1=kl$. A straightforward consequence of the above theorem is the following:  	
\begin{corr}	
$TL_{n}$ is semisimple when $q$ is generic or $l=2$ and $n$ is odd. The standard representations $\mathcal{V}_{n,p}$, $0\leq p\leq \lfloor\frac{n}{2}\rfloor$ then form a complete set of irreducibles.		
\end{corr}	
The possible labels $(n,p)$ of standard representations can be arranged into a Bratelli diagram. For roots of unity, the critical labels form into critical lines. If $(n,p)$ and $(n,p^{\prime})$ are located the same distance on either side of a given critical line then we refer to $(n,p),(n,p^{\prime})$ as a \textit{symmetric pair}. An example of the Bratelli diagram for $l=3$ is given in Figure \ref{Bratelli diagram}. Examples of symmetric pairs in this case include $(3,0),(3,1)$ and $(8,2),(8,4)$.
\begin{figure}
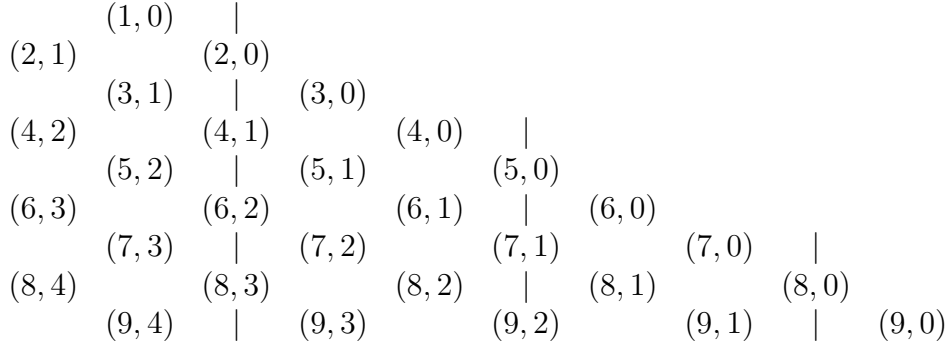
\label{Bratelli diagram}
\begin{align*}
\begin{array}{cccccccccc}&(1,0) & \lvert\\ (2,1)&  & (2,0)\\ & (3,1)& \lvert & (3,0)\\ (4,2)&  & (4,1) &  & (4,0)& \lvert\\ & (5,2) & \lvert & (5,1) &  & (5,0)\\ (6,3) &  & (6,2) &  & (6,1) & \lvert & (6,0)\\ & (7,3) & \lvert & (7,2) &  & (7,1) &  & (7,0) & \lvert\\ (8,4) &  & (8,3) &  & (8,2) & \lvert & (8,1) &  & (8,0)\\
& (9,4) & \lvert & (9,3) & & (9,2) & & (9,1) & \lvert & (9,0)
\end{array}
\end{align*}
\caption{The Bratelli diagram for $l=3$.}
\end{figure}	
We note that the importance of symmetric pairs is given in the following:
\begin{prop}\label{maps}
$End(\mathcal{V}_{n,p})\simeq\mathbb{C}$ for all values of $q$. $Hom(\mathcal{V}_{n,p},\mathcal{V}_{n,p^{\prime}})=0$ unless $q$ is a root of unity, $p^{\prime}>p$, and $(n,p),(n,p^{\prime})$ is a symmetric pair, in which case $Hom(\mathcal{V}_{n,p},\mathcal{V}_{n,p^{\prime}})\simeq\mathbb{C}$. There is an exceptional case for $\delta=0$, with $Hom(\mathcal{V}_{2,1},\mathcal{V}_{2,0})\simeq\mathbb{C}$.
\end{prop}	
For $p,p^{\prime}$ such that $Hom(\mathcal{V}_{n,p},\mathcal{V}_{n,p^{\prime}})\simeq\mathbb{C}$, we denote by $\phi_{n,p}\in Hom(\mathcal{V}_{n,p},\mathcal{V}_{n,p^{\prime}})$ the unique such map up to scalars. An explicit construction of this map was given in \cite{CGM} and is given as follows:
\begin{defin}\label{map def}
For $(n,0),(n,p)$ a symmetric pair, the map $\phi_{n,0}:\mathcal{V}_{n,0}\rightarrow\mathcal{V}_{n,p}$ is given by $\phi_{n,0}(\lvert_{n})=\sum a_{i}y_{i}$ where $y_{i}$ are link states in $\mathcal{V}_{n,p}$, and the sum is over all such link states. For the coefficients $a_{i}$, given a link state, $y_{i}$, each line $h$ in $y_{i}$ splits $y_{i}$ into two regions. Let $n(h)$ be the number of lines (including $h$ itself) in the region of $y_{i}$ not containing the right hand edge of the diagram. Then $a_{i}=(-1)^{e(y_{i})}\frac{[n-p]!}{\prod\limits_{h}[n(h)]}$, where the product is over all lines $h$ in $y_{i}$, and $e(y_{i})$ is the number of even quantum integers in the coefficient.
\end{defin}				
The maps $\phi_{n,p}:\mathcal{V}_{n
	,p}\rightarrow\mathcal{V}_{n,p^{\prime}}$ can be constructed from $\phi_{n-2p,0}$ by joining the image of $\phi_{n-2p,0}$ to the bottom of link states in $\mathcal{V}_{n,p}$. The maps $\phi_{n,p}$ allow us to almost fully describe the representations $\mathcal{L}_{n,p}$, $\mathcal{R}_{n,p}$. Namely, for a symmetric pair $(n,p),(n,p^{\prime})$, $Ker(\phi_{n,p})=\mathcal{R}_{n,p}$, $Im(\phi_{n,p})=\mathcal{R}_{n,p^{\prime}}$, so $\mathcal{L}_{n,p}\simeq\mathcal{R}_{n,p^{\prime}}$. Combined with knowledge of the $TL_{n}$ projective representations, the following can be shown:
\begin{prop}
For $\delta=0$ and $n$ even, the representations $\mathcal{L}_{n,p}$, $0\leq p\leq \frac{n}{2}-1$ form a complete set of irreducibles. Otherwise, the representations $\mathcal{L}_{n,p}$, $0\leq p\leq \lfloor\frac{n}{2}\rfloor$ form a complete set of irreducibles for $TL_{n}$.
\end{prop}		
Denote the dimension of $\mathcal{L}_{n,p}$ by $L_{n,p}$. For given $n,p,l$, we can write $n-2p+1=k(n,p)l+r(n,p)$ uniquely, where $k(n,p),r(n,p)\in\mathbb{N}$, $1\leq r(n,p)\leq l$. For roots of unity we have the following:
\begin{prop}\label{Recursion} The dimensions of the irreducible quotients $\mathcal{L}_{n,p}$ satisfy:
	\begin{align*}
	L_{n,p}=\left\{\begin{array}{cc} d_{n,p}& r(n,p)=l\\  L_{n-1,p}& r(n,p)=l-1\\  L_{n-1,p}+ L_{n-1,p-1}& otherwise\end{array}\right.
	\end{align*}
	with initial conditions $L_{n,0}=1$ and $L_{2p-1,p}=0$.
\end{prop}	
As $TL_{n}$ is semisimple for $q$ generic, the projective indecomposables are just the standard representations. At roots of unity, denote the projective cover of $\mathcal{L}_{n,p}$ by $\mathcal{P}_{n,p}$. Then they are as follows:
\begin{prop}
If $(n,p)$ is critical, so $\mathcal{L}_{n,p}\simeq\mathcal{V}_{n,p}$, then $\mathcal{P}_{n,p}\simeq\mathcal{V}_{n,p}$. If $(n,p)$ is not critical, and there is no $p^{\prime}>p$ such that $(n,p),(n,p^{\prime})$ is a symmetric pair, then $\mathcal{P}_{n,p}\simeq\mathcal{V}_{n,p}$. If such a symmetric pair exists, then $\mathcal{P}_{n,p}$ is the non-trivial extension $0\rightarrow\mathcal{V}_{n,p^{\prime}}\rightarrow\mathcal{P}_{n,p}\rightarrow\mathcal{V}_{n,p}\rightarrow 0$.
\end{prop}
A construction of the projective representations using idempotents was given in \cite{MartinBook}. We give an alternative construction in Section \ref{Extensions}.\\

The possible extensions of irreducible and standard representations, as well as a complete list of indecomposable $TL_{n}$ representations was classified in \cite{RSA}. 
\begin{prop}\label{Ext of TL_n}
Let $(n,p),(n,p^{\prime})$, and $(n,p^{\prime}),(n,p^{\prime\prime})$, $p<p^{\prime}<p^{\prime\prime}$, be symmetric pairs. Then the non-trivial extensions of irreducible and standard representations are as follows:
\begin{itemize}
	\item $Ext(\mathcal{L}_{n,p},\mathcal{L}_{n,p^{\prime}})\simeq Ext(\mathcal{L}_{n,p^{\prime}},\mathcal{L}_{n,p})\simeq\mathbb{C}$
	\item	$Ext(\mathcal{V}_{n,p},\mathcal{L}_{n,p^{\prime}})\simeq\mathbb{C}$
	\item $Ext(\mathcal{L}_{n,p},\mathcal{V}_{n,p^{\prime\prime}})\simeq\mathbb{C}$
	\item $Ext(\mathcal{L}_{n,p},\mathcal{V}_{n,p^{\prime}})\simeq\mathbb{C}$ if there is no $p_{-}<p$ such that $(n,p_{-}),(n,p)$ is symmetric.
	\item $Ext(\mathcal{V}_{n,p},\mathcal{V}_{n,p^{\prime}})\simeq Ext(\mathcal{V}_{n,p},\mathcal{V}_{n,p^{\prime\prime}})\simeq\mathbb{C}$
	\item For $\delta=0$, $Ext(\mathcal{V}_{2p+2,p},\mathcal{L}_{2p+2,p})\simeq\mathbb{C}$
	\end{itemize}
\end{prop}

\section{$TL_{\infty}(\delta)$}\label{TL infty}
Define the algebra $TL_{\infty}(\delta)$ as generated by $\{\mathbf{1},e_{i}\}$, $i\in\mathbb{N}$, satisfying the usual relations, where we only consider finite products of generators, and finite sums of basis elements.\\

\subsection{Finite dimensional representations}
We begin by considering finite dimensional representations of $TL_{\infty}$. The trivial representation is given by $e_{i}v=0$ for all $i$. For $\delta\neq \pm 1$, this is the only one-dimensional representation. However, for $\delta= \pm 1$, there is a second given by $e_{i}=\delta v$. For higher dimensions, we have the following result:

\begin{thm}\label{fin dim reps}
The only finite dimensional representations of $TL_{\infty}$ are direct sums of one-dimensional representations.
\end{thm}
To prove this, we first need the following two lemmas:
\begin{lemma}\label{dim of $TL_{n}$ reps}
Let $X$ be an indecomposable $TL_{n}(\delta)$ representation. Then either $X$ is one-dimensional, or $dim X\geq n-2$.	
\end{lemma}
\begin{proof}
We split this into several parts, and consider the cases $l=2$, $l=3$, and $n\leq 6$ separately. We then prove it inductively for $l\geq 4$ and $n\geq 6$. For irreducible representations $\mathcal{L}_{n,p}$, we denote $dim\mathcal{L}_{n,p}:=L_{n,p}$.\\
For $q$ generic, the irreducible representations have dimension $d_{n,p}$, for which $d_{n,0}=1$, and $d_{n,p}\geq n-1$ for $p\geq 1$. For $l=2$, for $n$ odd, the irreducibles are the standard representations, so $L_{n,p}=d_{n,p}$. For $n$ even, we have $r(n,p)=1$, so by the recursion relation in Proposition \ref{Recursion}, we have $L_{n,p}=d_{n-1,p}$.\\

For $l=3$, we have the cases $r(n,p)=1,2$, to consider. We first note that the case $r(n,p)=2$ implies the case $r(n,p)=1$. For the $r(n,p)=2$ case, $L_{n,1}=d_{n-2,1}+d_{n-4,0}=n-2$. For $L_{n,p}$, $2\leq p<\lfloor\frac{n}{2}\rfloor$, we can rewrite $L_{n,p}=L_{n-2,p-1}+d_{n-2,p}$ it follows from induction that if $L_{n-2,p-1}\geq n-4$, then $L_{n,p}\geq n-2$ as $d_{n,p}\geq 2$. For $p=\lfloor\frac{n}{2}\rfloor$, because of the condition $L_{2p-1,p}=0$, we instead get $L_{2p,p}=L_{2p+1,p}=1$.\\

For $n\leq 6$, the cases $l=4,5,6,7,8$, can be computed directly and checked. For $l\geq 9$, we note that as $0\leq p\leq \lfloor\frac{n}{2}\rfloor$, we have $1\leq r(n,p)\leq l-2$. Hence the only part of the recursion relation used is $L_{n,p}=L_{n-1,p}+L_{n-1,p-1}$. It follows that for $n\leq 6$ and $l\geq 9$, we have $L_{n,p}=d_{n,p}$.\\

Now we assume that	$l\geq 4$ and $n\geq 6$. Assume that n is minimal such that there is a $1\leq p\leq \lfloor\frac{n}{2}\rfloor$ such that $L_{n,p}<n-2$. Clearly $r(n,p)\neq 0$. If $r(n,p)=l-1$, then $L_{n,p}=L_{n-1,p}$, with $r(n-1,p)=l-2$. If $n-1=2p-1$, so that $L_{n-1,p}=0$, then $r(n-1,p)=0$, which requires $l=2$. Hence $L_{n-1,p}$ is non-zero, and we have $L_{n,p}=L_{n-2,p}+L_{n-2,p-1}$. If $n-2=2p-1$, so that $L_{n-2,p}=0$, then $r(n-2,p)=0$, which requires $l=3$. Hence $L_{n-2,p}$ and $L_{n-2,p-1}$ are non-zero. We assumed $n$ was minimal such that $L_{n,p}<n-2$. Hence we must have $L_{n-2,p},L_{n-2,p-1}\geq n-4$, and so $L_{n,p}\geq 2n-8$, and hence $L_{n,p}\geq n-2$ as $n\geq 6$.\\
If $1\leq r(n,p)\leq l-2$, then $L_{n,p}=L_{n-1,p}+L_{n-1,p-1}$. If $L_{n-1,p},L_{n-1,p-1}$ are both non-zero, then as $n$ is minimal, $L_{n-1,p},L_{n-1,p-1}\geq n-3$, and so $L_{n,p}\geq 2n-6\geq n-2$. If $n=2p$, so that $L_{n-1,p}=0$, we have $r(n-1,p-1)=2\leq l-2$ (as $l\neq 2,3$), and so $L_{n,p}=L_{n-1,p-1}=L_{n-2,p-1}+L_{n-2,p-2}\geq 2n-8\geq n-2$.\\

Hence we have shown that all irreducible $TL_{n}$ representations have dimension greater than or equal to $n-2$, except for the trivial representation, and $\mathcal{L}_{n,\lfloor\frac{n}{2}\rfloor}$ for $l=3$. Finally we note that the classification of extensions of irreducible representations in Proposition \ref{Ext of TL_n} that there is a non-trivial extension of two irreducible representations, if and only if they form a symmetric pair. Hence the only non-trivial extensions of one-dimensional representations are the extensions of $\mathcal{L}_{3,0},\mathcal{L}_{3,1}$ and $\mathcal{L}_{4,0},\mathcal{L}_{4,2}$ for $l=3$, and both these extensions have dimension $\geq n-2$.
\end{proof}
\begin{lemma}\label{diagonal entries}
Let $A$ be an $n\times n$ diagonal matrix with entries $a_{i}\in\{0,1\}$ or $\{0,-1\}$. Let $B$ be an arbitrary $n\times n$ matrix. If $ABA=A$, $BAB=B$, and $AB=BA$, then $A=B$.
\end{lemma}
\begin{proof}
Let $B:=(b_{ij})_{1\leq i,j\leq n}$. Considering $AB=BA$ we get $a_{i}b_{ij}=b_{ij}a_{j}$. Hence $b_{ij}=0$ unless $a_{i}=a_{j}$. Next, for $ABA=A$, we have $a_{i}^{2}b_{ii}=a_{i}$ and $a_{i}a_{j}b_{ij}=0$ for $i\neq j$. Hence if $a_{i}\neq 0$, then $b_{ii}=a_{i}$, and if $a_{i}a_{j}=1$ for $i\neq j$ then $b_{ij}=0$. Finally, considering $BAB=B$, we get $\sum\limits_{j=1}^{n}a_{j}b_{ij}b_{jk}=b_{ik}$. If $a_{i}=a_{k}=0$, but $a_{j}\neq 0$, we have $b_{ij}=b_{jk}=0$. Hence $b_{ik}=0$, and so we have $B=A$.
\end{proof}

Now we can prove Theorem \ref{fin dim reps}:
\begin{proof}
Assume that there is a finite dimensional $TL_{\infty}$ representation $X$, of dimension $k$. By Lemma \ref{diagonal entries}, if the generators $e_{1}$ and $e_{2}$ act as diagonal matrices with entries $\{0,1\}$ or $\{0,-1\}$, then every generator must act as the same diagonal matrix, and so $X$ must be a direct sum of one-dimensional representations. Assume otherwise, then consider the restriction of $X$ to a $TL_{n}$ representation for some $n>k+2$. Then by Lemma \ref{dim of $TL_{n}$ reps}, the only $TL_{n}$ representations in the decomposition of $X$ are one dimensional, and so the generators $e_{i}$, $1\leq i\leq n-1$, act as diagonal matrices, and hence $X$ is a direct sum of one-dimensional $TL_{\infty}$ representations.
\end{proof}

\subsection{$TL_{\infty}$ Representations generated by link states.}
We now aim to study representations of $TL_{\infty}$ by generalizing the standard representations of $TL_{n}$. We start by defining an \textit{infinite link state} as a link state with infinitely many points labelled by $n\in\mathbb{N}$. Given an infinite link state $w$, we define $s(w)\in\mathbb{N}\cup\{\infty\}$ as the number of strings on $w$, and $c(w)\in\mathbb{N}\cup\{\infty\}$ as the number of cups on $w$, we then define $\mathcal{X}(w)$ as the representation generated by $w$. We will generally prove results about $\mathcal{X}(w)$ by using results from the $TL_{n}$ representation theory. To do this, we will need the following definition:
\begin{defin}
	Define the \textit{restriction to length $n$} of a link-state as the link-state formed by deleting all points greater than $n$. Note that we do not allow the restriction to cut cups, as this will not always respect the bilinear form, as it introduces an extra string. In this case, we instead restrict to the smallest $n^{\prime}>n$ that doesn't require the cutting of a cup.
\end{defin}
It is clear that if two link states, $w,w^{\prime}$ differ at infinitely many points, then the representation generated by them must decompose into a direct sum, i.e. $\mathcal{X}(w,w^{\prime})\simeq\mathcal{X}(w)\oplus\mathcal{X}(w^{\prime})$. However it is not obvious if this is true when $w$ and $w^{\prime}$ only differ at finitely many points. To resolve this, we need the following: 
\begin{lemma}\label{map between link states}
	Let $x,y$ be two finite sums of infinite link states, that only differ up to the first $n$ points. If $x$ and $y$ both restrict to $\mathcal{V}_{n^{\prime},p}$, $x,y\notin\mathcal{R}_{n^{\prime},p}$, for $n^{\prime}\geq n$, then there are $a,b\in TL_{\infty}$ such that $x=ay$ and $y=bx$.
\end{lemma}
\begin{proof}
	This follows by restricting $x$ and $y$ to $\mathcal{V}_{n^{\prime},p}$ then using Lemma \ref{lemma 1}. The elements $a,b$ are given by extending the appropriate elements from Lemma $3.2$ to $TL_{\infty}$ by adding infinitely many strings to the right hand side of the diagrams. 
\end{proof}
In general, restrictions of link states are never in $\mathcal{R}_{n,p}$, with the exception when $\delta=0$ and $s(w)=0$. However in this case we can still find $TL_{\infty}$ elements that map link states to each other, as $\mathcal{R}_{n,p}$ is irreducible. From this, we get the following proposition:
\begin{prop}
Let $w$, $w^{\prime}$ be two infinite link states, and $\mathcal{X}(w,w^{\prime})$ the representation generated by them. If $w$ and $w^{\prime}$ differ at only the first $n$ points, and their restrictions are both in $\mathcal{V}_{n^{\prime},p}$ for $n^{\prime}\geq n$, then $\mathcal{X}(w,w^{\prime})\simeq\mathcal{X}(w)\simeq \mathcal{X}(w^{\prime})$. Otherwise, $\mathcal{X}(w,w^{\prime})\simeq \mathcal{X}(w)\oplus \mathcal{X}(w^{\prime})$.
\end{prop}
We can define an equivalence relation on link states as follows: We say $w\sim w^{\prime}$ if they differ at only finitely many points and there is an $n\in\mathbb{N}$ such that the restrictions of $w$ and $w^{\prime}$ are both in $\mathcal{V}_{m,p}$ for all $m\geq n$. From now on, we consider link states under this equivalence.\\
To prove irreducibility of representations, we generalize the method of \cite{RSA}, and define a subrepresentation generalizing the radical for each representation. We shall see that the representations are indecomposable, and their quotient by the radical is irreducible. We define the radical $\mathcal{R}(w)\subseteq\mathcal{X}(w)$ as follows:
\begin{align*}
\mathcal{R}(w):=\{x\in\mathcal{X}(w):\text{there is } n\in\mathbb{N}\text{ such that }x\text{ restricts to }\mathcal{R}_{n^{\prime},p}\text{ for all }n^{\prime}\geq n\}
\end{align*}
It is straightforward to see that $\mathcal{R}(w)$ is a $TL_{\infty}$ representation, as $\mathcal{R}_{n^{\prime},p}$ is always irreducible. Using Lemma \ref{map between link states}, Proposition $3.3$ of \cite{RSA} generalizes to give the following:
\begin{prop}\label{quotient by radical is irreducible}
$\mathcal{X}(w)$ is cyclic and indecomposable, and $\mathcal{X}(w)/\mathcal{R}(w)$ is irreducible.
\end{prop}
The only case when $\mathcal{X}(w)=\mathcal{R}(w)$ is when $\delta=0$ and $s(w)=0$, otherwise we have $\mathcal{R}(w)\subsetneq\mathcal{X}(w)$. The structure of $\mathcal{R}(w)$ is given by the following: 
\begin{prop}\label{radical irreducible or zero}
	$\mathcal{R}(w)$ is either zero or irreducible.
\end{prop}
\begin{proof}
	Given $x,y\in\mathcal{R}(w)$, then for $n$ large enough, $x$ and $y$ both restrict to $\mathcal{R}_{n,p}$. By Theorem $7.2$ of \cite{RSA}, $\mathcal{R}_{n,p}$ is either zero or irreducible. If it is zero for $n^{\prime}\geq n$, then $x$ and $y$ must be zero, and so $\mathcal{R}(w)$ is zero. If $\mathcal{R}_{n,p}$ is non-zero for all $n$, then as it is irreducible, we can find $TL_{n}$ elements that map the restrictions of $x$ and $y$ to each other. These extend to $TL_{\infty}$ elements that map $x$ and $y$ to each other, hence every $\mathcal{R}(w)$ element generates the representation, and so $\mathcal{R}(w)$ is irreducible.
\end{proof}
Combined with Proposition \ref{quotient by radical is irreducible}, it follows that every $\mathcal{X}(w)$ has at most two irreducible representations in its composition series.
We now want to determine when these representations are irreducible. For that, we will first need the following lemma:
\begin{lemma}\label{radical preserved}
	Let $x\in\mathcal{R}_{n,p}$, and $y\in\mathcal{V}_{2k,k}$, so that $y$ only consists of cups. Then $x\otimes y\in\mathcal{R}_{n+2k,p+k}$, where $x\otimes y$ is formed by joining $y$ onto the right hand side of $x$.
\end{lemma}
\begin{proof}
	As $\mathcal{R}_{n,p}\neq 0$, we know that $(n,p)$ is not critical. Then $(n+2k,p+k)$ is also non-critical, and so $\mathcal{R}_{n+2k,p+k}\neq 0$. Let $p^{\prime}<p$ be such that $(n,p^{\prime})$, $(n,p)$ form a symmetric pair, so that $Hom(\mathcal{V}_{n,p^{\prime}},\mathcal{V}_{n,p})\simeq\mathbb{C}$. The image of this map can be constructed explicitly as given in Definition \ref{map def}. Given a link state $z\in\mathcal{V}_{n,p^{\prime}}$, its image in $\mathcal{V}_{n,p}$ is the sum over all link states formed by joining together $2(p-p^{\prime})$ strings in $z$. The image of this, $Im(z)$, is then in $\mathcal{R}_{n,p}$. We now note that $Hom(\mathcal{V}_{n+2k,p^{\prime}+k},\mathcal{V}_{n+2k,p+k})\simeq\mathbb{C}$, as $(n+2k,p^{\prime}+k),(n+2k,p+k)$ is a symmetric pair if $(n,p^{\prime}),(n,p)$ is. Given $z\otimes y\in\mathcal{V}_{n+2k,p^{\prime}+k}$, we again construct its image, but this only depends on joining together strings, so $Im(z\otimes y)=Im(z)\otimes y$, and hence $Im(z)\otimes y\in\mathcal{R}_{n+2k,p+k}$. It then follows that this is true for any $x\in\mathcal{R}_{n,p}$. 	
\end{proof}
Note that if instead $y$ contains $j$ strings as well as cups in the above lemma, then it no longer holds, as $x\otimes y\in\mathcal{V}_{n+2k+j,p+k}$ which can be critical, and so we can't guarantee that $\mathcal{R}_{n+2k+j,p+k}\neq 0$. Given this, we can now classify which of the representations are irreducible.
\begin{thm}\label{Irreducibility of representations}
For $q$ generic, all representations $\mathcal{X}(w)$ generated by an infinite link state are irreducible. For $q$ a root of unity, with $l$ the smallest positive integer such that $q^{2l}=1$, if $s(w)=\infty$, $s(w)=kl-1$ for $k\in\mathbb{N}$, or $s(w)=0$ and $l=2$, then $\mathcal{X}(w)$ is irreducible. Otherwise it is indecomposable but not irreducible.
\end{thm}
\begin{proof}
For $q$ generic, $\mathcal{R}_{n,p}=0$ for all $n$ and $p$. Hence $\mathcal{R}(w)=0$, and so $\mathcal{X}(w)$ is irreducible.
 
For $q$ a root of unity, $\mathcal{R}_{n,p}$ may now be non-zero. We first look at the case $s(w)=\infty$. Assume there is some $x\in\mathcal{R}(w)$, $x\neq 0$. Then $x$ must restrict to $\mathcal{R}_{n,p}$ for $n$ large. However, as $s(w)=\infty$, we can always find an $n^{\prime}>n$ such that $x$ restricts to $\mathcal{V}_{n^{\prime},p^{\prime}}$ with $(n^{\prime},p^{\prime})$ critical. Hence $x=0$, and therefore $\mathcal{R}(w)=0$.

For $s(w)=kl-1$, any $x\in\mathcal{X}(w)$ will restrict to $\mathcal{V}_{n,\frac{n-kl+1}{2}}$ for $n$ large enough, which is critical and so has trivial radical. 

For $l=2$ and $s(w)=0$, $\mathcal{X}(w)=\mathcal{R}(w)$, which we have already shown is irreducible.

For the remaining representations, by Proposition \ref{quotient by radical is irreducible} we know the representations are indecomposable. For a given link state $w$ with $s(w)=j$, given $x\in\mathcal{X}(w)$ that restricts to $\mathcal{R}_{n,\frac{n-j}{2}}$ for $n$ large enough, by Lemma \ref{radical preserved}, it will restrict to $\mathcal{R}_{n+2k,\frac{n-j}{2}+k}$ for $k\in\mathbb{N}$. Hence $\mathcal{R}(w)$ is non-empty. However, not every element of $\mathcal{X}(w)$ is in $\mathcal{R}(w)$, namely $w\notin\mathcal{R}(w)$. Therefore the $\mathcal{X}(w)$ are not irreducible.
\end{proof}
Now that we have determined when these representations are irreducible, we want to find out if they are pairwise isomorphic. To do this, we look at when there are non-zero morphisms between representations. To start, we first need the following definition:
\begin{defin}
Given two infinite link states, $w,z$, that differ at only finitely many points, with $s(w)$ and $s(z)$ both finite, $s(w)<s(z)$, we say $w$ and $z$ form a symmetric pair if $(s(z),0),(s(z),\frac{s(z)-s(w)}{2})$ form a symmetric pair.
\end{defin}
\begin{prop}\label{maps 2}
Given two infinite link states, $w,z$, if $w$ and $z$ are equivalent link states, or $w$ and $z$ form a symmetric pair with $s(z)<s(w)$, then $Hom(\mathcal{X}(w),\mathcal{X}(z))\simeq\mathbb{C}$. Otherwise, $Hom(\mathcal{X}(w),\mathcal{X}(z))=0$.
\end{prop}
\begin{proof}
We consider two cases, when $w$ and $z$ differ at finitely many points, and when they differ at infinitely many points. When $w$ and $z$ only differ at finitely many points, assume there is some map $\phi_{w,z}:\mathcal{X}(w)\rightarrow\mathcal{X}(z)$. We can restrict this to some $n$ large enough to get a map $\phi_{n,p}:\mathcal{V}_{n,p}\rightarrow\mathcal{V}_{n,p^{\prime}}$. By Proposition \ref{maps}, either $\phi_{n,p}$ is the identity and $p=p^{\prime}$, or $\phi_{n,p}$ is as in Definition \ref{map def} and $(n,p),(n,p^{\prime})$ is a symmetric pair. By considering all possible restrictions to $n^{\prime}>n$, we see that either $\phi_{w,z}$ is the identity with $w$ and $z$ equivalent, or $(n^{\prime},p+\frac{n^{\prime}-n}{2}),(n^{\prime},p^{\prime}+\frac{n^{\prime}-n}{2})$ are symmetric pairs for all $n^{\prime}>n$. As the construction of each of the $\phi_{n^{\prime},p}$ depend only on the number of strings, we can take $p=0$ to get $(s(z),0),(s(z),\frac{s(z)-s(w)}{2})$ is a symmetric pair.

If $w$ and $z$ differ at infinitely many points, assume there is some map $f:\mathcal{X}(w)\rightarrow\mathcal{X}(z)$, $f:w\mapsto\sum a_{i}z_{i}$, $a_{i}\in\mathbb{C}$, $z_{i}\in\mathcal{X}(z)$. As $w$ and $z$ differ at infinitely many points, we can find some point $i$ such that there is a simple cup at the $i$th position on one of $w$ or $f(w)$ but not the other. Considering $(e_{i}-\delta 1)f(w)$, we see that $f$ can't exist for $\delta\neq 0$. If $\delta=0$, we can instead consider $e_{i}f(w)$, unless one of $w$ or $z$ is constructed from the other by replacing multiple simple cups with two strings. In this case, let $i,i+1,j,j+1$ be the locations of two adjacent replacements. We further assume that a cup connects the $(i+2)$th and $(j-1)$th points. Acting $e_{i+1}e_{j-1}$ on $w$ and $f(w)$ we then get zero for whichever has strings at $i,i+1$ and $j,j+1$, but for whichever has cups we get:
\begin{figure}[H]
	\centering
	\includegraphics[width=0.7\linewidth]{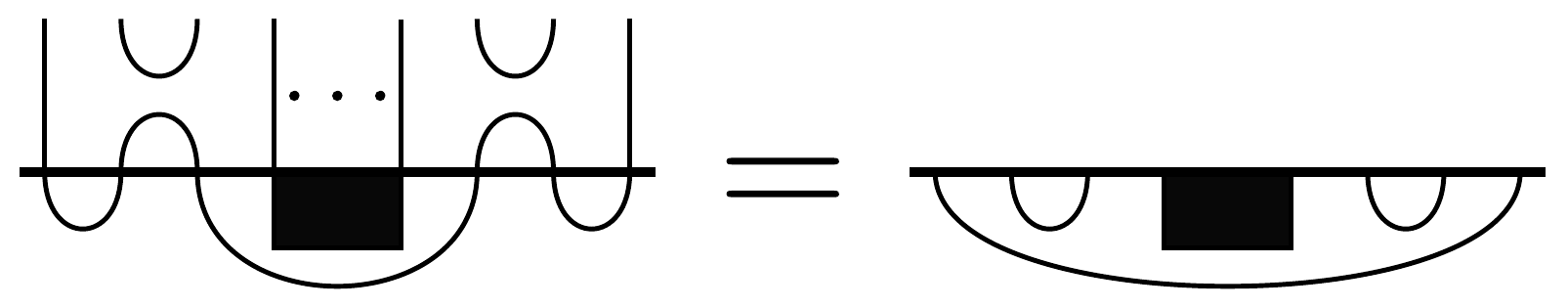}
\end{figure}
Hence $f$ must equal zero.
\end{proof}
\begin{corr}\label{iso}
$\mathcal{X}(w)\simeq\mathcal{X}(z)$ if and only if $w$ and $z$ are equivalent link states.
\end{corr}
The map $\phi_{w,z}:\mathcal{X}(w)\rightarrow\mathcal{X}(z)$ described in Proposition \ref{maps 2} for when $w$ and $z$ are a symmetric pair is an immediate generalization of the map defined in Definition \ref{map def}. As previously noted, Proposition \ref{quotient by radical is irreducible} and Proposition \ref{radical irreducible or zero} combined give that when $\mathcal{X}(w)$ is not irreducible, it has exactly two irreducible components in its composition series, namely $\mathcal{R}(w)$ and $\mathcal{L}(w):=\mathcal{X}(w)/\mathcal{R}(w)$. We now want to determine when these irreducible components are isomorphic. Given an infinite link state $w$ with $s(w)=j_{0}\in\mathbb{N}$, let $w^{(j)}$ be an infinite link state that differs at finitely many points from $w$, and $s(w^{(j)})=j_{0}+j$, i.e. the link state formed by cutting $\frac{j}{2}$ cups on $w$. For $0\leq j_{0}\leq l-2$, we then have the following sequence:
\begin{align*}
...\xrightarrow{\phi_{3}}\mathcal{X}(w^{(j_{2})})\xrightarrow{\phi_{2}}\mathcal{X}(w^{(j_{1})})\xrightarrow{\phi_{1}}\mathcal{X}(w)
\end{align*} 
with $j_{k}:=2(kl-1)-j_{k-1}$, and $\phi_{i}\in Hom(\mathcal{X}(w^{(j_{i})}),\mathcal{X}(w^{(j_{(i-1)})}))$. By considering restricting terms in this sequence to finite $n$, we get the following: 
\begin{prop}\label{exact}
	The above sequence is exact, with $Im(\phi_{k})\simeq\mathcal{R}(w^{(j_{k-1})})$.	
\end{prop}
We note that this allows each irreducible representation to be realized as the kernel of the bilinear form, except for irreducibles $\mathcal{L}(w)$, $s(w)<l-1$, noting that for $l=2$, $\mathcal{L}(w)=0$ if $s(w)=0$. By the proof of Proposition \ref{maps 2} we find that the irreducibles are not isomorphic for $s(w)>l-1$. For $s(w)<l-1$, we have the following result:
\begin{prop}\label{no iso}
For $l\geq 4$, given inequivalent infinite link states $w,z$ with $s(w)<l-1$ and $s(z)$ finite, then $\mathcal{L}(w)$ and $\mathcal{L}(z)$ are not isomorphic.	
\end{prop}	
\begin{proof}
If $w$ and $z$ differ at only finitely many points then the proof follows immediately from restriction. Hence, we only consider when $w$ and $z$ differ at infinitely many points. Assume otherwise that $\mathcal{L}(w)\simeq\mathcal{L}(z)$, with $w^{\prime}\mapsto z^{\prime}$, where we view $w^{\prime}$,$z^{\prime}$ as elements of $\mathcal{X}(w)$, $\mathcal{X}(z)$ respectively. Choose some $n$ large such that $e_{n}w^{\prime}=\delta w^{\prime}$, but $e_{n}z^{\prime}\neq \delta z^{\prime}$, and the $n$th and $(n+1)$th points of $z^{\prime}$ are connected to cups, not strings. For this map to give an isomorphism of $\mathcal{L}(w)$ and $\mathcal{L}(z)$, we must have $e_{n}z^{\prime}-\delta z^{\prime}\in\mathcal{R}(z)$. We can therefore restrict $e_{n}z^{\prime}-\delta z^{\prime}$ to $\mathcal{R}_{n^{\prime},p}$ for some $n^{\prime}>n$. Hence we must have $\langle e_{n}z^{\prime}-\delta z^{\prime},z^{\prime}\rangle=0$, i.e. $\langle e_{n}z^{\prime},z^{\prime}\rangle=\delta\langle z^{\prime},z^{\prime}\rangle$. As there isn't a cup at the $n$th point of $z^{\prime}$, but it is  connected to a cup, we find that $\langle e_{n}z^{\prime},z^{\prime}\rangle=\delta^{-1}\langle z^{\prime},z^{\prime}\rangle$, (see figure below), and so $e_{n}z^{\prime}-\delta z^{\prime}\notin\mathcal{R}(z)$ unless $\delta^{2}=1$. Hence there are no such isomorphism maps for $l\geq 4$. 	
\end{proof}	
\begin{figure}[H]
	\centering
	\includegraphics[width=0.5\linewidth]{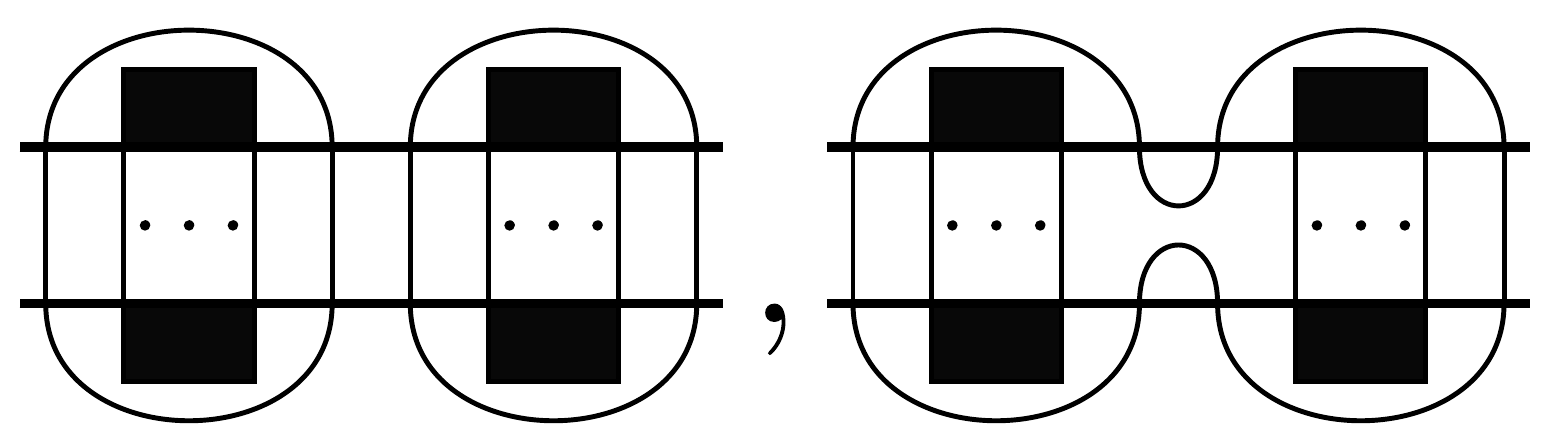}
	\caption{How $e_{n}$ affects the bilinear form.}
\end{figure}
For the cases $s(w)=0,1$, $l=3$, we have the following result:
\begin{corr}\label{one dim}
	For $l=3$, let $X$ be the non-trivial one-dimensional representation given by $e_{i}\nu=\delta\nu$. If $s(w)=0,1$, then $\mathcal{L}(w)\simeq X$.
\end{corr}
\begin{proof}
If $l=3$ and $s(w)=0,1$, then $\mathcal{R}(w)$ is the set of elements of the form $w^{\prime}-\delta w^{\prime\prime}$, where $w^{\prime},w^{\prime\prime}$ are link states in $\mathcal{X}(w)$ such that joining together two strings on $w^{\prime}$ and cutting a cup results in $w^{\prime\prime}$. Any two link states in $\mathcal{X}(w)$ can be transformed into each other by a sequence of cutting cups and joining together strings, hence quotienting by $\mathcal{R}(w)$ results in a one-dimensional representation, which we can see is not the trivial one.
\end{proof}
Combining together the results of Propositions \ref{maps 2}, \ref{exact}, \ref{no iso}, and Corollaries \ref{iso} and \ref{one dim}, we get the following:
\begin{thm}\label{Irr iso}
The Representations $\mathcal{X}(w),\mathcal{X}(w^{\prime})$ are isomorphic if and only if the link states $w,w^{\prime}$ are equivalent. The non-trivial irreducible quotients $\mathcal{L}(w),\mathcal{L}(w^{\prime})$ are isomorphic if and only if the link states $w,w^{\prime}$ are equivalent, or $l=3$ and $s(w)=0,1$.
\end{thm}	 
Now that we have determined the structure of the $\mathcal{X}(w)$ representations, we want to study if there are any possible extensions of them. We define a method of constructing the projective indecomposable $TL_{n}$ representations that generalizes to give extensions of certain $TL_{\infty}$ representations: 
\begin{defin}\label{Extensions}
Let $(n,p),(n,p^{\prime})$ be a symmetric pair, with $p<p^{\prime}$, and $x\in\mathcal{V}_{n,p}$. Let $\phi$ be the unique map $\phi:\mathcal{V}_{n,p}\rightarrow\mathcal{V}_{n,p^{\prime}}$ given by $\phi:x\mapsto\sum\limits_{i}a_{i}y_{i}$, $a_{i}\in\mathbb{C}$, where the sum is over all elements $y_{i}\in\mathcal{V}_{n,p^{\prime}}$ formed by adding $\frac{p^{\prime}-p}{2}$ cups to $x$ (see Definition \ref{map def}). We define an extension $\mathcal{V}_{n,p^{\prime}}\rightarrow P_{n,p}\rightarrow\mathcal{V}_{n,p}$ as follows: If a generator $e_{k}$ increases the number of cups of $x$, set $\tilde{e}_{k}x:=\sum\limits_{j}a_{j}y_{j}$, where the sum is over all elements $y_{j}\in\phi(x)$ with a simple cup at the $k$th position, and coefficients from $\phi(x)$.
\end{defin}
As an example, the map $\phi:\mathcal{V}_{5,0}\rightarrow\mathcal{V}_{5,2}$ exists for $l=4$, i.e. $\delta=\pm\sqrt{2}$. This map is given as follows:
\begin{figure}[H]
	\centering
	\includegraphics[width=1\linewidth]{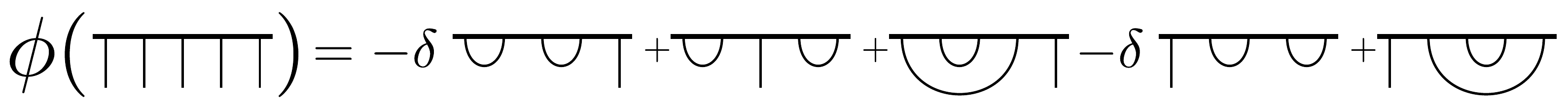}
\end{figure}
From this, we define the action of the generators in the extension $\mathcal{V}_{5,2}\rightarrow P_{5,0}\rightarrow\mathcal{V}_{5,0}$ as follows:
\begin{figure}[H]
	\centering
	\includegraphics[width=1\linewidth]{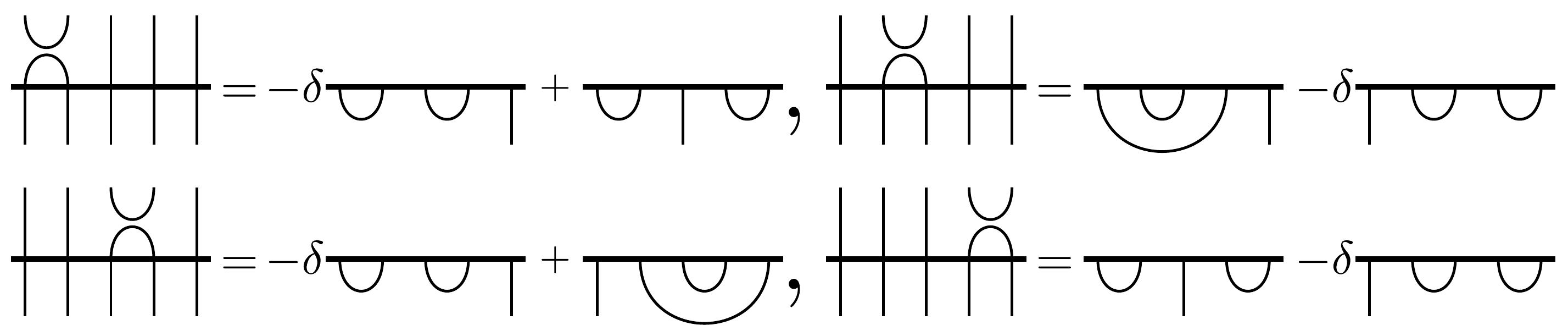}
\end{figure}
\begin{remark}
For some extensions, such as $\mathcal{V}_{n,1}\rightarrow P_{n,0}\rightarrow\mathcal{V}_{n,0}$, the simpler choice of taking all coefficients to be equal to one is possible, but for higher $p$ this does not always define an extension of representations, for example in the case $\mathcal{V}_{6,2}\rightarrow P_{6,0}\rightarrow\mathcal{V}_{6,0}$. We shall show that our more complicated choice of coefficients always defines a non-split extension.
\end{remark}
\begin{thm}
This extension is non-split, i.e. $P_{n,p}\simeq\mathcal{P}_{n,p}$.
\end{thm}
\begin{proof}
We consider the case $\mathcal{V}_{n,p}\rightarrow P_{n,0}\rightarrow\mathcal{V}_{n,0}$, as any other case can be constructed from this by adding cups to both sides. We first check that this new action defines a representation. The relation $e_{k}^{2}x=\delta e_{k}x$ is immediate. For the relation $e_{k}e_{k\pm 1}e_{k}x=e_{k}x$, we have $e_{k}e_{k+1}\cup\downarrow=e_{k}\downarrow\cup=\cup\downarrow$, where we use $\downarrow$ to denote a point connected to an arbitrary string or cup. 
For the relation $e_{i}e_{j}=e_{j}e_{i}$ if $\lvert i-j\rvert>1$, this is immediate if $p=1$. For $p>1$, fix such a choice of $i$ and $j$, then rewrite $\phi$ as follows:
$\phi(\lvert^{\otimes n})=\sum\limits_{k_{m}}a_{k_{1}}y_{k_{1}}+a_{k_{2}}y_{k_{2}}+a_{k_{3}}y_{k_{3}}+a_{k_{4}}y_{k_{4}}+a_{k_{5}}y_{k_{5}}$, where each of the $y_{k_{m}}$ are diagrams of the form as follows: $y_{k_{1}}$ are diagrams with a simple cup at the $i$th position but not at the $j$th position. $y_{k_{2}}$ are diagrams with a simple cup at the $j$th position but not at the $i$th position. $y_{k_{3}}$ are diagrams with simple cups at both the $i$th and $j$th position. $y_{k_{4}}$ are diagrams of the form:
\begin{figure}[H]
	\centering
	\includegraphics[width=0.3\linewidth]{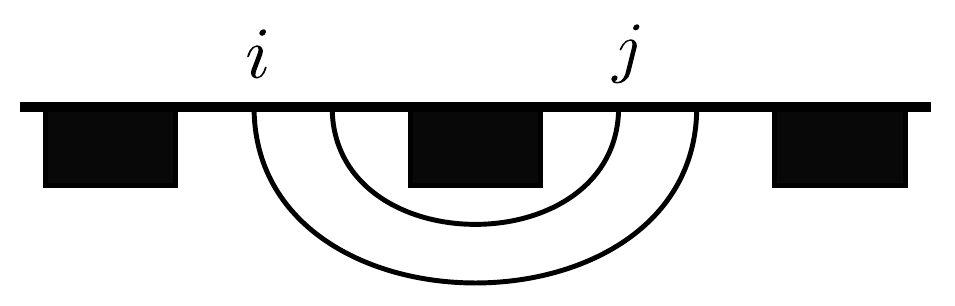}
\end{figure}
The $y_{k_{5}}$ are then the diagrams not appearing in any of the previous sums. Under this relabelling, we have $\tilde{e}_{i}\lvert^{\otimes n}=\sum\limits_{k_{m}}a_{k_{1}}y_{k_{1}}+a_{k_{3}}y_{k_{3}}$, $\tilde{e}_{j}\lvert^{\otimes n}=\sum\limits_{k_{m}}=a_{k_{2}}y_{k_{2}}+a_{k_{3}}y_{k_{3}}$. We want to show that $e_{j}\tilde{e}_{i}\lvert^{\otimes n}=e_{i}\tilde{e}_{j}\lvert^{\otimes n}$, which we can now rewrite as $\sum\limits_{k_{1}}a_{k_{1}}e_{j}y_{k_{1}}=\sum\limits_{k_{2}}a_{k_{2}}e_{i}y_{k_{2}}$. We now note that $e_{i}y_{k_{4}}=e_{j}y_{k_{4}}\in y_{k_{3}}$, and the $y_{k_{4}}$ are the only diagrams without a simple cup at $i$ such that $e_{j}y_{k_{4}}$ has a simple cup at $i$, and respectively for $i$ and $j$ switched.
Next we consider $e_{i}\phi(\lvert^{\otimes n})=0$, in terms of the sum, this gives:
\begin{align*}
&\sum\limits_{k_{m}}\delta a_{k_{1}}y_{k_{1}}+a_{k_{2}}e_{i}y_{k_{2}}+\delta a_{k_{3}}y_{k_{3}}+a_{k_{4}}e_{i}y_{k_{4}}+a_{k_{5}}e_{i}y_{k_{5}}=0
\end{align*} 
By separating into diagrams with or without a simple cup at the $j$th position, we get:
\begin{align*}
&\sum\limits_{k_{m}}a_{k_{2}}e_{i}y_{k_{2}}+\delta a_{k_{3}}y_{k_{3}}+a_{k_{4}}e_{i}y_{k_{4}}=0
\end{align*}
Repeating for $e_{j}\phi(\lvert^{\otimes n})$, we get:
\begin{align*}
&\sum\limits_{k_{m}}a_{k_{1}}e_{j}y_{k_{1}}+\delta a_{k_{3}}y_{k_{3}}+a_{k_{4}}e_{j}y_{k_{4}}=0
\end{align*}
As $\sum\limits_{k_{4}}a_{k_{4}}e_{i}y_{k_{4}}=\sum\limits_{k_{4}}a_{k_{4}}e_{j}y_{k_{4}}$, this gives $\sum\limits_{k_{2}}a_{k_{2}}e_{i}y_{k_{2}}=\sum\limits_{k_{1}}a_{k_{1}}e_{j}y_{k_{1}}$, and hence the relation holds.\\

To show that the extension doesn't split, we again consider the case $\mathcal{V}_{n+1,p}\rightarrow P_{n+1,0}\rightarrow\mathcal{V}_{n+1,0}$, and show there is no injective map $\mathcal{V}_{n+1,0}\rightarrow P_{n+1,0}$, with the general case following from this. We start with the case $p=1$. Note that the map $\phi:\mathcal{V}_{n,0}\rightarrow\mathcal{V}_{n,1}$ can be written as $\phi(\lvert^{\otimes n})=\sum\limits_{i=1}^{n-1}(-1)^{i+1}[i]c_{i}$, where $c_{i}$ is the diagram with a single cup at the $i$th position, so $\tilde{e}_{i}\lvert^{\otimes n}=(-1)^{i+1}[i]c_{i}$ in $ P_{n+1,0}$. Assume there is some non-zero map $\Psi:\lvert^{\otimes n+1}\mapsto\tilde{\lvert}^{\otimes n+1}+\sum\limits_{i=1}^{n} b_{i}c_{i}$, where we use $\tilde{\lvert}^{\otimes n+1}$ to denote $\lvert^{\otimes n+1}\in P_{n+1,0}$, $b_{i}\in\mathbb{C}$. By considering the action of $e_{j}$ on $\Psi$ for each $j$, we find that the coefficients $b_{i}$ must satisfy the following:
	\begin{align*}
	1+\delta b_{1}+b_{2}&=0,& (-1)^{i+1}[i]+b_{i-1}+\delta b_{i}+b_{i+1}&=0,& (-1)^{n+1}[n]+b_{n-1}+\delta b_{n}&=0
	\end{align*}
The last equation can be considered as the second with the extra condition $b_{n+1}=0$. By induction, we find the following general formula for $b_{k}$:
\begin{align*}
b_{k}&=(-1)^{k+1}\left(\sum\limits_{i=0}^{\lfloor\frac{k-1}{2}\rfloor}(k-2i-1)[k-2i-1]\right)+(-1)^{k+1}[k]b_{1}
\end{align*}
The condition $b_{n+1}=0$ then gives:
\begin{align*}
b_{1}&=-\sum\limits_{i=0}^{\lfloor\frac{n}{2}\rfloor}\frac{(n-2i)[n-2i]}{[n+1]}
\end{align*}
For $(n+1,0),(n+1,1)$ to be a symmetric pair, we must have $n+1=kl$, and so $[n+1]=0$. Hence there is no solution for the coefficients, unless they can be simplified so that $[n+1]$ doesn't appear on the denominator. Using the solution for $b_{1}$, we get that $b_{n}$ is of the form:
\begin{align*}
b_{n}&=\frac{(-1)^{n}}{[n+1]}\left(\sum\limits_{i=0}^{\lfloor\frac{n}{2}\rfloor}(n-2i)[n-2i][n]-\sum\limits_{j=0}^{\lfloor\frac{n-1}{2}\rfloor}(n-2j-1)[n-2j-1][n+1]\right)
\end{align*}
By using the relations $[n-2i][n]-[n-2i-1][n+1]=[2i+1]$, and $\sum\limits_{i=0}^{j}[2i+1]=[j+1]^{2}$, this simplifies to give:
\begin{align*}
b_{n}&=\frac{(-1)^{n}}{[n+1]}\left(\sum\limits_{i=1}^{n}[i]^{2}\right)
\end{align*}
At roots of unity, we have $[k]\in\mathbb{R}$, and so $\sum\limits_{i=1}^{n}[i]^{2}>0$. Hence $b_{n}$ can't be simplified to remove $[n+1]$, and so the coefficients have no solution.\\

For $p>1$, we first need the following lemma:
\begin{lemma}
Given the maps $\phi_{p}:\mathcal{V}_{n,0}\rightarrow\mathcal{V}_{n,p}$, $\phi_{p-1}:\mathcal{V}_{n-1,0}\rightarrow\mathcal{V}_{n-1,p-1}$. Let $\bar{\phi}_{p}\in\mathcal{V}_{n-1,p-1}$ be formed from the image of $\phi_{p}$ by taking only terms with a cup at the $n$th point, and deleting the $n$th point. Then $\bar{\phi}_{p}=\pm\phi_{p-1}(\lvert^{\otimes n-1})$.
\end{lemma}
\begin{proof}
Acting $e_{i}$, $1\leq i\leq n-2$, on the link states contributing to $\bar{\phi}_{p}$, we get either zero or a sum of link states with a string at the $n$th point. If the result has a string at the $n$th point, then deleting the $n$th point results in link states with $p$ cups which are zero in $\mathcal{V}_{n-1,p-1}$. Hence $e_{i}\bar{\phi}_{p}=0$ for $1\leq i\leq n-2$, and so $\bar{\phi}_{p}\in\mathcal{R}_{n-1,p-1}$. As $\mathcal{R}_{n-1,p-1}$ is the image of $\phi_{p-1}$, it is one-dimensional, hence we have $\bar{\phi}_{p}=\lambda\phi_{p-1}(\lvert^{\otimes n-1})$ for some non-zero constant $\lambda$. To find $\lambda$, consider the coefficients in $\phi_{p}$, $\phi_{p-1}$, respectively of the following two diagrams:
\begin{figure}[H]
 	\centering
 	\includegraphics[width=0.5\linewidth]{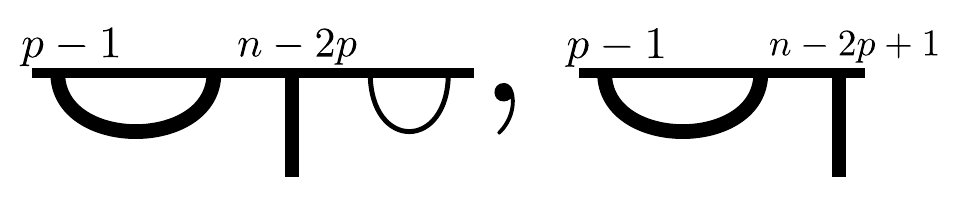}
\end{figure}	
The coefficients of these diagrams are $\frac{(-1)^{n-p+1}}{[n-p]}$, $1$. Hence we must have $\bar{\phi}_{p}=\frac{(-1)^{n-p+1}}{[n-p]}\phi_{p-1}(\lvert^{\otimes n-1})$. For $(n,0),(n,p)$ to be a symmetric pair, we must have $n-p+1=kl$, and so $[n-p]=\pm 1$.
\end{proof}	
To show the extension doesn't split for $p>1$, assume there is some minimal $p$ such that the map $\Psi_{p}:\mathcal{V}_{n,0}\rightarrow P_{n,p}$ exists. Let $\{y_{i}\}$ be the link states appearing in the image of $\Psi_{p}$ that have their $n$th point connected to a cup, and let $\{b_{i}\}$ be the corresponding coefficients of the link states. Denote by $\downarrow y_{i}$ the link state in $\mathcal{V}_{n-1,p-1}$ formed from $y_{i}$ by cutting the cup connected to the $n$th point and then removing the $n$th point. Consider the map $\theta_{p-1}:\lvert^{\otimes n-1}\mapsto \tilde{\lvert}^{\otimes n-1}\pm\sum b_{i}\downarrow y_{i}$, it follows from the previous lemma that if $\Psi_{p}$ exists, then $\theta_{p-1}$ exists, and defines a map $\theta_{p-1}:\mathcal{V}_{n-1,0}\rightarrow P_{n-1,p-1}$, but repeating this, we get a map $\theta_{1}:\mathcal{V}_{n-p+1,0}\rightarrow P_{n-p+1,0}$, which we have shown doesn't exist. Hence $\Psi_{p}$ doesn't exist, and so the extension doesn't split. As the only non-split extension of symmetric pairs of standard representations is the projective indecomposable representation, we find $P_{n,p}\simeq\mathcal{P}_{n,p}$.

\end{proof}
We can use this extension to construct a second extension of representations.
\begin{corr}
Let $(n,p^{\prime}),(n,p^{\prime\prime})$ and $(n,p),(n,p^{\prime})$ be symmetric pairs, with $p^{\prime\prime}>p^{\prime}>p$, and $\phi:\mathcal{V}_{n,p}\rightarrow\mathcal{V}_{n,p^{\prime}}$. Then there is a non split extension $V_{n,p^{\prime\prime}}\rightarrow \mathcal{Q}_{n,p}\rightarrow\mathcal{V}_{n,p}$, given by $\bar{e}_{i}x:=\tilde{e}_{i}\phi(x)$, $x\in\mathcal{V}_{n,p}$.
\end{corr}
It follows immediately from the above construction, that given infinite link states $w,w^{\prime},w^{\prime\prime}$ that differ at only finitely many points, and such that $s(w)>s(w^{\prime})>s(w^{\prime\prime})$ are both finite and form two symmetric pairs, there are non-split extensions of representations $\mathcal{X}(w^{\prime})\rightarrow \mathcal{P}(w)\rightarrow\mathcal{X}(w)$, $\mathcal{X}(w^{\prime\prime})\rightarrow \mathcal{Q}(w)\rightarrow\mathcal{X}(w)$.
\begin{prop}
Given infinite link states $w$, $z$ that differ at only finitely many points with $s(w),s(z)$ finite. Then the above two extensions are the only possible extensions of $\mathcal{X}(w)$, $\mathcal{X}(z)$.
\end{prop}
\begin{proof}
Assume otherwise, so that there are link states $w,z$ with a non-split extension $\mathcal{X}(w)\rightarrow M\rightarrow\mathcal{X}(z)$. As $w$ and $z$ differ at only finitely many points, let $xz^{\prime}=w^{\prime}$ be a relation in the extension, with $x\in TL_{\infty}$, and $w^{\prime}$, $z^{\prime}$ viewed as elements of $\mathcal{X}(w)$, $\mathcal{X}(z)$ respectively. We can then restrict $w^{\prime},z^{\prime}$ to finite n-link states and $x$ to a $TL_{n}$ element, for some $n$. As $s(w)$ is finite, then either the restriction splits, and so does the extension $M$, or else the restriction is the extension of standard representations defined above, and so the extension $M$ is the same construction.
\end{proof}
\begin{remark}
We have only considered extensions of link state representations with finitely many strings. However the case when there are infinitely many strings is also non-trivial. For example, Let $w_{i}$ be the infinite link state, up to equivalence, with $i$ cups and infinitely many strings. We can construct an extension $\mathcal{X}(w_{1})\rightarrow M\rightarrow\mathcal{X}(w_{0})$, by setting $e_{j}w_{0}:=c_{j}$, where $c_{j}$ is the infinite link state with a single cup at the $i$th position. For $TL_{n}$, this generally splits, but for $TL_{\infty}$, this extension is non-split for all $q$, as defining an injective map $\mathcal{X}(w_{0})\rightarrow M$ would require a sum with infinitely many terms. It gets more complicated when we consider $\mathcal{X}(w_{2})\rightarrow M \rightarrow\mathcal{X}(w_{0})$. Let $c_{i,j}$ be the diagram given as follows:
\begin{figure}[H]
	\centering
	\includegraphics[width=1\linewidth]{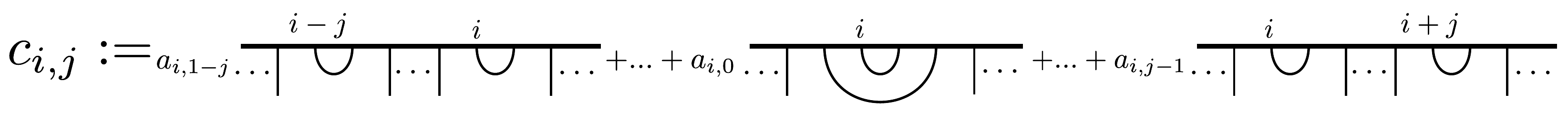}
\end{figure}
Then setting $e_{i}w_{0}:=c_{i,j}$, with the condition $e_{i}e_{k}w_{0}=e_{k}e_{i}w_{0}$ for $\lvert i-k\rvert>1$ on the coefficients $a_{i,j}\in\mathbb{C}$, this defines an extension for each $j\geq 1$. By considering the action of $e_{j+1}e_{i}$ for different extensions, it suggests that these may not be isomorphic for different $j$. Hence in general extensions of link state representations with infinitely many strings appears to be more complicated. We leave a full discussion of their extensions for future research.
\end{remark}

\subsection{A bilinear form on the Spin Chain Representation}\label{spin chain}
A standard result due to Jimbo and Martin \cite{Jimbo, MartinSW} is that $TL_{n}(q+q^{-1})$ is in Schur-Weyl duality with $U_{q}(\mathfrak{sl}_{2})$ over the spin-chain representation $(\mathbb{C}^{2})^{\otimes n}$. $\mathbb{C}^{2}$ is viewed as the 2-dimensional $U_{q}(\mathfrak{sl}_{2})$ representation $X_{1}:=\{\nu_{-1},\nu_{1}\}$, with action $K\nu_{i}=q^{-i}\nu_{i}$, $E\nu_{-1}=0$, $E\nu_{1}=\nu_{-1}$, $F\nu_{-1}=\nu_{1}$, $F\nu_{1}=0$. The $U_{q}(\mathfrak{sl}_{2})$ action on $(\mathbb{C}^{2})^{\otimes n}$ is then given by use of the coproduct. The Temperley-Lieb action commuting with the $U_{q}(\mathfrak{sl}_{2})$ action is given by $e(\nu_{-1}\otimes\nu_{-1})=e(\nu_{1}\otimes\nu_{1})=0$, $e(\nu_{-1}\otimes\nu_{1})=q\nu_{-1}\otimes\nu_{1}-\nu_{1}\otimes\nu_{-1}$, $e(\nu_{1}\otimes\nu_{-1})=q^{-1}\nu_{1}\otimes\nu_{-1}-\nu_{-1}\otimes\nu_{1}$. For $q$ generic, the spin chain decomposes as a $U_{q}(\mathfrak{sl}_{2})\otimes TL_{n}(q+q^{-1})$ bimodule as follows:
\begin{align*}
(\mathbb{C}^{2})^{\otimes n}\simeq \bigoplus\limits_{\substack{i=0,\\ i=n+1\mod 2}}^{n+1}X_{i}\otimes\mathcal{V}_{n,\frac{n-i+1}{2}}
\end{align*}
where $X_{i}$ is the irreducible $(i+1)$-dimensional $U_{q}(\mathfrak{sl}_{2})$ representation with basis $X_{i}:=\{\nu_{j}:-i\leq j\leq i, j=i\mod 2\}$, $K\nu_{j}=q^{-j}\nu_{j}$, $E\nu_{j}=[\frac{i-j}{2}+1]\nu_{j-2}$, $E\nu_{-i}=0$, $F\nu_{j}=[\frac{i+j}{2}+1]\nu_{j+2}$, $F\nu_{i}=0$.
There is a $U_{q}(\mathfrak{sl}_{2})\otimes TL_{n}(q+q^{-1})$ invariant orthogonal bilinear form on the spin chain defined as follows:
\begin{align*}
\langle \nu_{i_{1}}\otimes...\otimes\nu_{i_{n}},\nu_{j_{1}}\otimes...\otimes\nu_{j_{n}}\rangle:=q^{\#(\nu_{-1})\#(\nu_{1})}\delta_{i_{1},j_{1}}...\delta_{i_{n},j_{n}}
\end{align*}
where $\#(\nu_{\pm 1})$ is the number of $\nu_{\pm 1}$'s appearing in $\nu_{i_{1}}\otimes...\otimes\nu_{i_{n}}$, and $i_{k},j_{k}\in\{\pm 1\}$. We note that the determinant of the Gram matrix of this bilinear form being non-zero doesn't prove semisimplicity of the spin chain representation. This is due to the failure of lemma $3.2$ of \cite{RSA} to generalize. For example, take the basis $\{\nu_{-1}\otimes\nu_{-1},\nu_{-1}\otimes\nu_{1},q^{-1}\nu_{-1}\otimes\nu_{1}+\nu_{1}\otimes\nu_{-1},\nu_{1}\otimes\nu_{1}\}$ of $(\mathbb{C}^{2})^{\otimes 2}$ with $q=\pm i$. Then the determinant of the Gram matrix with respect to this basis is $q^{2}$. However taking $x:=\nu_{-1}\otimes\nu_{1}$, $y:=q^{-1}\nu_{-1}\otimes\nu_{1}+\nu_{1}\otimes\nu_{-1}$, we have $\langle x,y\rangle x=x$, but there is no $U_{q}(\mathfrak{sl}_{2})\otimes TL_{n}(q+q^{-1})$ element that maps $y\mapsto x$.  
The relationship between the above bilinear form and the one on the standard representations is given by the following:
\begin{prop}
The restriction of the bilinear form on the spin chain to a standard representation induces the bilinear form on the standard representation.
\end{prop}
\begin{proof}
We assume that $q$ is generic, as the map $(\mathbb{C}^{2})^{\otimes n}\rightarrow \mathcal{V}_{n,p}$ is independent of $q$. Consider the copy of $\mathcal{V}_{n,p}$ appearing in the highest weight space of $(\mathbb{C}^{2})^{\otimes n}$, i.e. $\nu_{n-2p+1}\otimes\mathcal{V}_{n,p}$. A single cup is given by $\cup:=q^{-1}\nu_{1}\otimes\nu_{-1}-\nu_{-1}\otimes\nu_{1}$, and a string by $\nu_{1}$. It follows that if the cup is connected to two strings then the bilinear form gives zero. Hence the bilinear form is the same as the one on the standard representation. The case for lower weights follows from the commutativity of the Temperley-Lieb generators with the $E$ action.
\end{proof}
Similarly, restricting the bilinear form on the spin chain to one of the $U_{q}(\mathfrak{sl}_{2})$ representations in its decomposition induces a bilinear form on $X_{i}$, given as follows:
\begin{align*}
\langle \nu_{j},\nu_{k}\rangle:=\delta_{j,k}\frac{[i+1]!}{[\frac{i-j}{2}+1]![\frac{i+j}{2}+1]!}
\end{align*}
	
We want to generalize Schur-Weyl duality to the case of $TL_{\infty}$. To do this, we need an appropriate generalization of the spin chain representation. However, simply taking the tensor product of infinitely many copies of $X_{1}$ results in issues with the $U_{q}(\mathfrak{sl}_{2})$ action. Instead we aim to generalize the decomposition of $(\mathbb{C}^{2})^{\otimes n}$. Let $w_{0}$, $w_{1}$ be link states with $s(w_{0})=0$, $s(w_{1})=1$, and let $w_{k}^{(2i)}$, $k=0,1$, $i\in\mathbb{N}$ be a link state up to equivalence that differs from $w_{k}$ at only finitely many points and $s(w_{k}^{(2i)})=k+2i$. The spin chain $(\mathbb{C}^{2})^{\otimes n}$ can be decomposed as given previously, but with the standard representations relabelled in terms of the number of strings. There is a natural inclusion from $n$ to $n+2$ under this labelling, so we end up with two decompositions in the limit $n\rightarrow\infty$, one odd, one even. We define the spin chain generalization $\mathcal{S}(w_{k})$, $k=0,1$, by replacing the standard representations labelled by a certain string number in the decomposition with the representation $\mathcal{X}(w_{k}^{2i})$ that has the same string number. For generic $q$ we can give this generalization explicitly:
\begin{align*}
\mathcal{S}(w_{k}):=&\{\nu_{j}\otimes x : x\in\mathcal{X}(w_{k}^{(2i)}), -k-2i\leq j\leq k+2i, j=k \mod 2\}\\
t(\nu_{j}\otimes x):=&\nu_{j}\otimes tx, \:\:\: t\in TL_{\infty}
\end{align*}
The $U_{q}(\mathfrak{sl}_{2})$ action on this is then given as follows:
\begin{align*}
K(\nu_{i}\otimes x):=&q^{-i}\nu_{i}\otimes x,& E(\nu_{i}\otimes x):=&[\frac{s(x)-i}{2}+1]\nu_{i-2}\otimes x,& E(\nu_{-s(x)}\otimes x):=&0\\ & & F(\nu_{i}\otimes x):=&[\frac{s(x)+i}{2}+1]\nu_{i+2}\otimes x,& F(\nu_{s(x)}\otimes x):=&0
\end{align*}
The case when $q$ is a root of unity is more complicated. In this case, $\mathcal{S}(w_{k})$ has the same basis, but now both the $TL_{\infty}$ and $U_{q}(\mathfrak{sl}_{2})$ actions are more difficult to give a general formula for. To describe the actions, we can use the decomposition of the $TL_{n}$ spin chain at roots of unity given in \cite{GV}, along with the basis of $U_{q}(\mathfrak{sl}_{2})$ projective representations given in \cite{BGT}. As an example, for $l=2$, the even spin chain can be decomposed as follows:
\begin{equation*}
\begin{tikzcd}
& & \vdots & & \\
\nu_{-4}\otimes\mathcal{V}_{4} \arrow[rd, "\tilde{F}"] & \nu_{-2}\otimes\mathcal{V}_{4} \arrow[l, "E"'] \arrow[r, "F"] \arrow[d, dashed, "\tilde{e}_{i}"] \arrow[rd, "\tilde{F}", near end] & \nu_{0}\otimes\mathcal{V}_{4}  \arrow[ld, "\tilde{E}"', near end] \arrow[rd, "\tilde{F}", near end] & \nu_{2}\otimes\mathcal{V}_{4} \arrow[l, "E"'] \arrow[r, "F"] \arrow[d, dashed, "\tilde{e}_{i}"] \arrow[ld, "\tilde{E}"', near end] & \nu_{4}\otimes\mathcal{V}_{4}  \arrow[ld, "\tilde{E}"'] \\
& \nu_{-2}\otimes\mathcal{V}_{2} \arrow[rd, "\tilde{F}"] & \nu_{0}\otimes\mathcal{V}_{2} \arrow[l, "E"'] \arrow[r, "F"] \arrow[d, dashed, "\tilde{e}_{i}"] & \nu_{2}\otimes\mathcal{V}_{2}  \arrow[ld, "\tilde{E}"'] & \\
& & \nu_{0}\otimes\mathcal{V}_{0} & &
\end{tikzcd}
\end{equation*}
where we have relabelled $\mathcal{V}_{n-2p}:=\mathcal{V}_{n,p}$ the standard representations in terms of the number of strings. The map $\tilde{e}_{i}$ denotes the extension defined in Definition \ref{Extensions}. The maps $\tilde{E}$, $\tilde{F}$ are defined by $\tilde{E}(\nu_{i}\otimes x):=\nu_{i-2}\otimes\phi(x)$, $\tilde{F}(\nu_{i}\otimes x):=\nu_{i+2}\otimes\phi(x)$. We have also omitted labelling the action of the divided powers $E^{(2)}:=\frac{E}{[2]!}$, $F^{(2)}:=\frac{F}{[2]!}$ which act as $E^{(2)},F^{(2)}:\nu_{i}\otimes x\mapsto \nu_{i\pm 2}\otimes x$ for all $\nu_{i}\otimes x$. Given this decomposition, the spin chain generalization $\mathcal{S}(w_{0})$ for $l=2$ is given as follows:
\begin{equation*}
\begin{tikzcd}
& & \vdots & & \\
\nu_{-4}\otimes\mathcal{X}(w_{0}^{(4)}) \arrow[rd, "\tilde{F}"] & \nu_{-2}\otimes\mathcal{X}(w_{0}^{(4)}) \arrow[l, "E"'] \arrow[r, "F"] \arrow[d, dashed, "\tilde{e}_{i}"] \arrow[rd, "\tilde{F}", near start] & \nu_{0}\otimes\mathcal{X}(w_{0}^{(4)})  \arrow[ld, "\tilde{E}"', near start] \arrow[rd, "\tilde{F}", near start] & \nu_{2}\otimes\mathcal{X}(w_{0}^{(4)}) \arrow[l, "E"'] \arrow[r, "F"] \arrow[d, dashed, "\tilde{e}_{i}"] \arrow[ld, "\tilde{E}"', near start] & \nu_{4}\otimes\mathcal{X}(w_{0}^{(4)})  \arrow[ld, "\tilde{E}"'] \\
& \nu_{-2}\otimes\mathcal{X}(w_{0}^{(2)}) \arrow[rd, "\tilde{F}"] & \nu_{0}\otimes\mathcal{X}(w_{0}^{(2)}) \arrow[l, "E"'] \arrow[r, "F"] \arrow[d, dashed, "\tilde{e}_{i}"] & \nu_{2}\otimes\mathcal{X}(w_{0}^{(2)})  \arrow[ld, "\tilde{E}"'] & \\
& & \nu_{0}\otimes\mathcal{X}(w_{0}) & &
\end{tikzcd}
\end{equation*}
By its construction, the action of $TL_{\infty}$ and $U_{q}(\mathfrak{sl}_{2})$ commute on this representation. Given this spin chain generalization, and a choice of link states $w_{0},w_{1}$, we can define the functor $\mathcal{F}_{w_{0},w_{1}}:Rep U_{q}(\mathfrak{sl}_{2})\rightarrow Rep TL_{\infty}$ by $\mathcal{F}(-):=Hom(-,\mathcal{S}(w_{0})\oplus\mathcal{S}(w_{1}))$. Let $\mathcal{C}(w_{0},w_{1})$ be the Serre subcategory generated by finitely generated $TL_{\infty}(q+q^{-1})$ representations generated by link states that differ from $w_{0}$ or $w_{1}$ at only finitely many points. We make the following conjecture:
\begin{conj}
The functor $\mathcal{F}_{w_{0},w_{1}}$ defines an equivalence of abelian categories between the category of finite dimensional $U_{q}(\mathfrak{sl}_{2})$ representations and $\mathcal{C}(w_{0},w_{1})$.
\end{conj}	
\subsection*{Acknowledgements}
This work was supported by ISF grants 2095/15 and 711/18 and the Center for Advanced Studies in Mathematics in Ben Gurion University. The author thanks Inna Entova for her comments and for suggesting the project.

\bibliographystyle{abbrv}
\bibliography{ref}

\end{document}